\documentclass[11pt,reqno]{amsart}

\usepackage{amsthm,amsmath,amssymb,amsfonts,amscd,mathtools}
\RequirePackage{array}
\RequirePackage{cancel}
\RequirePackage{enumerate}
\RequirePackage{mathtools}
\RequirePackage{tikz-cd}
\usetikzlibrary{positioning}
\usepackage{dynkin-diagrams} % CTAN: dynkin-diagrams
\RequirePackage{geometry}
\RequirePackage{comment}
\usepackage[colorlinks = true,
            linkcolor = blue,
            urlcolor  = blue,
            citecolor = blue]{hyperref}
\usepackage{cleveref}

\allowdisplaybreaks

\newcommand{\R}{\mathbb R}
\newcommand{\C}{\mathbb C}

\newcommand{\D}{\mathbb {D}}
\renewcommand{\P}{\mathbb {P}}

\newcommand{\G}{\mathsf {G}}
\renewcommand{\H}{\mathsf {H}}

\newcommand{\g}{\mathfrak g}

\renewcommand{\b}{\mathfrak b}

\newcommand{\q}{\mathfrak q}
\newcommand{\qj}{\mathfrak q_J}
\newcommand{\h}{\mathfrak h}
\renewcommand{\sl}{\mathfrak {sl}}

\newcommand{\Q}{\mathcal {Q}}

\newcommand{\J}{\mathcal J}
\newcommand{\T}{\mathcal T}
\newcommand{\M}{\mathcal M}
\newcommand{\CS}{\mathcal S}

\renewcommand{\i}{\operatorname{i}}
\newcommand{\Ad}{\operatorname{Ad}}
\newcommand{\ad}{\operatorname{ad}}
\newcommand{\Id}{\operatorname{Id}}
\newcommand{\tr}{\operatorname{tr}}
\newcommand{\Aut}{\operatorname{Aut}}

\renewcommand{\Im}{\operatorname{Im}}

\newcommand{\rank}{\operatorname{rank}}
\newcommand{\I}{\operatorname{I}}
\newcommand{\II}{\operatorname{II}}
\newcommand{\III}{\operatorname{III}}
\newcommand{\SL}{\mathsf{SL}}
\newcommand{\PSL}{\mathsf{PSL}}

\newcommand{\Gr}{\operatorname{Gr}}
\newcommand{\St}{\operatorname{St}}
\newcommand{\End}{\operatorname{End}}
\newcommand{\Hom}{\operatorname{Hom}}
\newcommand{\quo}{{\operatorname{quo}}}

\newcommand{\cand}{{\operatorname{cand}}}
\newcommand{\flip}{{\operatorname{flip}}}
\newcommand{\std}{{\operatorname{std}}}

\newcommand{\rest}[2]{\left.{#1}\right|_{#2}}

\theoremstyle{plain}
\newtheorem{theorem}{Theorem}[section]
\newtheorem{proposition}[theorem]{Proposition}
\newtheorem{lemma}[theorem]{Lemma}
\newtheorem{corollary}[theorem]{Corollary}

\theoremstyle{definition}
\newtheorem{example}[theorem]{Example}

\theoremstyle{remark}
\newtheorem{remark}[theorem]{Remark}

\title{Invariant complex structures on $\SL(2,\C)$}
\author{Joseph Kwong}
\address[J. Kwong]{The University of Queensland, Brisbane, QLD 4072, Australia}
\email{j.kwong@student.uq.edu.au}

\author{Asia Mainenti}
\address[A. Mainenti]{Institute of Mathematics “Simion Stoilow” of the Romanian Academy, 21 Calea Grivitei, 010702 Bucharest, Romania}
\email{asia.mainenti@imar.ro; asia.mainenti@gmail.com}


\begin{document}

\begin{abstract}
    We classify the left-invariant complex structures on the real Lie group $\SL(2,\C)$   up to Lie group automorphism. The classification consists of
    the bi-invariant complex structure, a family parameterised by the unit disk, and one new, non-regular complex structure. The topology  of the space of left-invariant complex structures (modulo automorphisms) is not Hausdorff. Finally, we show that the corresponding complex manifolds always admit left-invariant balanced metrics.
\end{abstract}

\maketitle

\section{Introduction}

An integrable complex structure $J$ on a real Lie group $\G$ is called \emph{left-invariant} if every left translation of $\G$ is a biholomorphism. Every even-dimensional semisimple Lie group admits a left-invariant complex structure \cite{morimoto_1956}. In real dimension six, the classification of left-invariant complex structures on semisimple Lie groups follows from \cite{snow_1986}, with the exception of $\SL(2,\C)$ (and its quotient $\PSL(2,\C) = \SL(2,\C) / \{\pm I_2\}$). The following theorem settles this remaining case:

\begin{theorem}
    \label{main theorem}
    The left-invariant complex structures (up to Lie group automorphism) on the real Lie group $\SL(2,\C)$ consist of the bi-invariant complex structure, a family of regular complex structures parameterised by the open unit disk, and one non-regular complex structure.
\end{theorem}

Here, a left-invariant complex structure is called \textit{regular} if it is right-invariant with respect to a Cartan subgroup $\H$ of $\G$ (see Section \ref{subsection_regular}). The bi-invariant complex structure and the regular family already appear in \cite{snow_1986}; Theorem \ref{main theorem} asserts the existence and uniqueness (up to automorphism) of a non-regular complex structure.

We note that the same classification holds for $\PSL(2,\C)$, because  every Lie group automorphism of $\SL(2,\C)$ descends to $\PSL(2,\C)$ (see Remark \ref{remark_group_automorphism}).

Topologically, the quotient space is the disk adjoined with two isolated points: the bi-invariant complex structure and the non-regular complex structure (see Theorem \ref{theorem topology}). However, the quotient space is not Hausdorff: every neighbourhood of the origin of the disk contains the non-regular complex structure.

A Hermitian metric $g$ on a complex manifold $(M^{2n}, J)$ is called \textit{balanced} if the $2$-form $\omega (\cdot,\cdot) := g(J\cdot,\cdot)$ satisfies $d \omega^{n-1} = 0$, and is called \textit{pluriclosed} if $\partial \overline \partial \omega = 0$. 
\begin{theorem}
    \label{theorem hermitian metrics}
    For any left-invariant complex structure $J$ and any cocompact lattice $\Gamma$ of $\SL(2,\C)$, the compact complex manifold $(\Gamma \backslash \SL(2,\C), J)$ admits a balanced metric. 
\end{theorem}

On the other hand, we show that $(\Gamma \backslash \SL(2,\C), J)$ does not admit any pluriclosed metrics (see Section \ref{section_balanced_pluriclosed}). This verifies the Fino--Vezzoni conjecture \cite{FV} for $(\Gamma \backslash \SL(2,\C), J)$.

The paper is organised as follows. In Section \ref{section_complex_structures}, we recall how left-invariant complex structures on a Lie group are encoded by subalgebras of the complexified Lie algebra, and we translate equivalence under Lie group automorphisms into this language (Section \ref{subsection_general_equivalence}). In Section \ref{secsl}, we specialise to the real Lie algebra $\sl(2,\C)$: we identify its complexification with $\mathfrak{sl}(2,\mathbb C)\oplus\mathfrak{sl}(2,\mathbb C)$ (Section \ref{subsection_complexification}), describe its real automorphisms (Section \ref{subsection_real_automorphisms}), and record some elementary facts about its (complex) Borel subalgebras (Section \ref{subsection_borel}). 

Section \ref{section_main_theorem} contains the proof of Theorem \ref{main theorem}, which is restated more explicitly in \Cref{theorem classification}.
In Section \ref{section_topology}, we  describe the topology of the space of left-invariant complex structures modulo automorphisms (see Theorem \ref{theorem topology}). Finally, in Section \ref{section_balanced_pluriclosed}, we prove the statements about Hermitian metrics made above.

\smallskip
{\bf Acknowledgments.} 
This collaboration started during the MATRIX Research Program "Analytic and Geometric Methods on Complex Manifolds" and was pursued while the second named author was visiting The University of Queensland, which she warmly thanks for the hospitality. 
J. Kwong was supported by an Australian Government Research Training Scholarship.
A. Mainenti was partly supported by the PNRR-III-C9-2023-I8 grant CF 149/31.07.2023 {\it Conformal Aspects of Geometry and Dynamics}.

\section{Invariant complex structures on Lie groups}
\label{section_complex_structures}
Let $\G$ be a real Lie group. 
Recall that a complex structure on $\G$ is a (smooth) endomorphism of the tangent bundle $J:T\G \rightarrow T\G$ such that $J^2 = -\Id_{T\G}$ and $J$ satisfies the following integrability condition:
$$0= [X,Y] + J [JX,Y] + J [X,JY] - [JX, JY]$$
for all smooth vector fields $X,Y$ on $\G$.
We say that a complex structure $J$ on $\G$ is \textit{left-invariant} if all left translations $L_g:\G \rightarrow \G$ given by $L_g(h) = gh$ are biholomorphisms with respect to $J$, i.e. $dL_g \circ J  = J \circ dL_g$ for all $g \in \G$. 

Let $\g$ be the Lie algebra of $\G$, which is the real vector space of all left-invariant vector fields on $\G$, equipped with the Lie bracket of vector fields. 
If $J$ is a left-invariant complex structure on $\G$ and $X \in \g$, then $JX \in \g$. 
Thus, $J$ restricts to a linear map $J: \g \rightarrow \g$ which satisfies $J^2 = - \Id_\g$ and $N_J \equiv 0$, where $N_J: \g \times \g \rightarrow \g$ is given by 
$$N_J(X,Y) = [X,Y] + J [JX,Y] + J [X,JY] - [JX, JY]$$
for all $X,Y \in \g$. Conversely, every such linear map $J \in \End(\g)$ extends uniquely to a left-invariant complex structure on $\G$. Thus, we may identify the set of all left-invariant complex structures on $\G$ with 
$$\J(\g) := \left\{\text{$J \in \End(\g)$ such that $J^2 = - \Id_\g$ and $N_J \equiv 0$}\right\}.$$

Let $\g^\C$ denote the complexification of $\g$. Recall that $\g^\C$ is a complex Lie algebra: the underlying real vector space of $\g^\C$ is $\g^\C = \g \oplus \hat \i \g$, where $\hat \i$  is a formal symbol. We write elements of $\g^\C$ as $X + \hat \i Y$ for $X,Y \in \g$. Multiplication by $\i$ is given by $\i \cdot (X + \hat \i Y) := - Y + \hat \i X$, and the Lie bracket is obtained by $\C$-linearly extending the Lie bracket on $\g \oplus \hat \i 0 \cong \g$. 
The complexification $\g^\C$ comes equipped with a complex conjugation map $\sigma:\g^\C \rightarrow \g^\C$, given by $\sigma(X + \hat \i Y) = X - \hat \i Y$.
Next, let  $\Gr_n(\g^\C)$ denote the Grassmannian of $n$-dimensional complex subspaces of $\g^\C$, where $\dim_\C \g^\C = 2n$, and consider the following subset:
\begin{align*}
    \Q(\g) :=& \left\{\text{$\q \in \Gr_n(\g^\C)$ such that $\q \oplus \sigma \q = \g^\C$ and $\q$ is a subalgebra of $\g^\C$}\right\}.
\end{align*}

For $J \in \J(\g)$, the  $\i$-eigenspace $\qj$ of the $\C$-linear extension $J: \g^\C \rightarrow \g^\C$ belongs to $\Q(\g)$. Indeed,  $J^2 = - \Id_\g$ implies that $\qj \oplus \sigma \qj = \g^\C$, and $N_J \equiv 0$ implies that $\qj$ is a subalgebra. Conversely, any $\q \in \Q(\g)$ is the $\i$-eigenspace of a unique $J \in \J(\g)$.
In summary:
\begin{proposition}
    \label{prop_bijection_J_Q}
    We have  a bijection $\J(\g) \cong \Q(\g)$ given by
    $$\J(\g) \longrightarrow \Q(\g), \qquad J \longmapsto \text{\normalfont the $\i$-eigenspace of $J:\g^\C \rightarrow \g^\C$}.$$
\end{proposition}

\begin{example}
    A complex structure on $\G$ is called \textit{bi-invariant} if both the left and right translations of $\G$ are biholomorphisms. For a left-invariant complex structure $J \in \J(\g)$ with $\i$-eigenspace $\q \in \Q(\g)$, the following are equivalent:
    \begin{enumerate}[(i)]
        \item $J$ is bi-invariant.
        \item $\ad(X) \circ J  = J \circ \ad(X)$ for all $X \in \g$.
        \item $[\g^\C,\q] \subseteq \q$, i.e. $\q$ is an ideal of $\g^\C$. 
    \end{enumerate}

\end{example}

\subsection{Equivalence of invariant complex structures} 
\label{subsection_general_equivalence}
Let $\Aut(\G)$ denote the Lie group automorphisms of $\G$, and let $\Aut_\R(\g)$ denote the (real) Lie algebra automorphisms of $\g$.
\begin{proposition}
    \label{prop_equivalences}
    Let $J,J' \in \J(\g)$, and let $\q, \q' \in \Q(\g)$ denote their corresponding $\i$-eigenspaces, respectively. If $\G$ is simply connected, the following are equivalent:
    \begin{enumerate}[\normalfont (i)]
        \item\label{peq1} There exists $F \in \Aut(\G)$ such that $J' = dF \circ J \circ dF^{-1}$.
        \item\label{peq2} There exists $\varphi \in \Aut_\R(\g)$ such that $J' = \varphi \circ J \circ \varphi^{-1}$.
        \item\label{peq3} There exists $\varphi \in \Aut_\R(\g)$ such that $\q' = \varphi(\q)$. Here, we have identified $\varphi$ with its $\C$-linear extension $\varphi:\g^\C \rightarrow \g^\C$.
    \end{enumerate}
\end{proposition}
If any, and hence all, of these conditions hold, we say that $J$ and $J'$ are \textit{equivalent}, and we say that $\q$ and $\q'$ are \textit{equivalent}. 

\begin{remark}
    Complex-linear extension $\Aut_\R(\g) \xhookrightarrow{} \Aut_\C(\g^\C)$ is an injective Lie group homomorphism. We identify $\Aut_\R(\g)$ with its image, which is the subgroup of automorphisms commuting with $\sigma$:
    $$\Aut_\R(\g) \cong \left\{ \varphi \in \Aut_\C(\g^\C) : \varphi \circ \sigma = \sigma \circ \varphi\right\}.$$
\end{remark}

\begin{proof}[Proof of Proposition \ref{prop_equivalences}]
    Clearly, \eqref{peq1} implies \eqref{peq2}: if $F \in \Aut(\G)$, then the restriction to left-invariant vector fields $dF|_\g: \g \rightarrow \g$ is a Lie algebra automorphism.
    
    Next, consider the differential map $d:\Aut(\G) \rightarrow \Aut_\R(\g)$, which is a Lie group homomorphism. When $\G$ is connected, this map is injective. When $G$ is simply connected, $d:\Aut(\G) \rightarrow \Aut_\R(\g)$ is also surjective. Thus, for every $\varphi \in \Aut_\R(\g)$, there exists $F \in \Aut(\G)$ such that $\varphi = dF$, so \eqref{peq2} implies \eqref{peq1}.

    The group $\Aut_\R(\g)$ acts on $\J(\g)$ and $\Q(\g)$ by 
    $$\varphi \cdot J = \varphi \circ J \circ \varphi^{-1}, \qquad \varphi \cdot \q = \varphi(\q),$$
    respectively. Observe that \eqref{peq2} holds if and only if $J$ and $J'$ belong to the same $\Aut_\R(\g)$ orbit, and \eqref{peq3} holds if and only if $\q$ and $\q'$ belong to the same $\Aut_\R(\g)$ orbit. Now, the bijection (see Proposition \ref{prop_bijection_J_Q}) 
    $$\J(\g) \longrightarrow \Q(\g), \qquad J \longmapsto \q_J :=\text{\normalfont the $\i$-eigenspace of $J$}$$
    is $\Aut_\R(\g)$-equivariant: indeed, observe that $\q_{\varphi \circ J \circ \varphi^{-1}} = \varphi(\q)$. Therefore, \eqref{peq2} and \eqref{peq3} are equivalent.
\end{proof}

\subsection{Regular complex structures}
\label{subsection_regular}
Let $\G$ be a real semisimple Lie group with Lie algebra $\g$. A subalgebra $\h$ of $\g$ is called a \textit{Cartan subalgebra} if $\h^\C$ is a Cartan subalgebra of $\g^\C$ (see \cite[Chapter VI.6]{knapp}). A connected Lie subgroup $\H$ of $\G$ is called a \textit{Cartan subgroup} if its Lie algebra is a Cartan subalgebra of $\g$. 
\begin{proposition}
    \label{prop_general_regularity}
    Let $J \in \J(\g)$ and let $\q \in \Q(\g)$ be the corresponding $\i$-eigenspace. The following are equivalent:
    \begin{enumerate}[\normalfont (i)]
        \item $J$ is right-invariant with respect to a Cartan subgroup $\H$ of $\G$.
        \item There exists a Cartan subalgebra $\h$ of $\g$ such that for every $X \in \h$, $$\ad(X) \circ J = J \circ \ad(X).$$ 
        \item There exists a Cartan subalgebra $\h$ of $\g$ such that $[\h^\C, \q] \subseteq \q$.
    \end{enumerate}
\end{proposition}
If any of the equivalent conditions hold, we say that $J$ and $\q$ are \textit{regular}. Clearly, any bi-invariant complex structure $J \in \J(\g)$ is regular.

\section{The Lie algebra \texorpdfstring{$\sl(2,\C)$}{sl(2,C)}}
\label{secsl}
\subsection{The complexification of \texorpdfstring{$\sl(2,\C)$}{sl(2,C)}}
\label{subsection_complexification}
Let $\SL(2,\C)$ be the real Lie group of all complex $2$ by $2$ matrices $A \in \C^{2 \times 2}$ with $\det(A) = 1$. As a smooth manifold, $\SL(2,\C)$ is diffeomorphic to $S^3 \times \R^3$ \cite[Theorem 6.31 (c)]{knapp}, and so is simply connected.
We identify the Lie algebra of $\SL(2,\C)$ with its tangent space at the identity, which is 
$$\sl(2,\C) = \left\{ X \in \C^{2 \times 2} : \tr X = 0\right\}.$$
Under this identification, the induced Lie bracket on $\sl(2,\C)$ is the commutator bracket $[X,Y] = XY - YX$. 
Let $H,E,F \in \sl(2,\C)$ denote the matrices given by
    $$H = \begin{pmatrix}
    1 & 0 \\ 0 & -1
\end{pmatrix}, \qquad E =\begin{pmatrix}
    0 & 1\\ 0 & 0
\end{pmatrix},  \qquad  F = \begin{pmatrix}
    0 & 0 \\ 1 & 0
\end{pmatrix}.$$

The complexification $\sl(2,\C)^\C$ of the real Lie algebra $\sl(2,\C)$ is isomorphic to the complex Lie algebra $\sl(2,\C) \oplus \sl(2,\C)$. 
In fact, we have an isomorphism of complex Lie algebras 
$$
\sl(2,\C)^\C \rightarrow \sl(2,\C) \oplus \sl(2,\C), \qquad X + \hat \i Y \mapsto (X + \i Y, \overline{X - \i Y}),
$$
where $\overline{(\cdot)}$  denotes entrywise complex conjugation in $\C^{2 \times 2}$. Henceforth, we identify $\sl(2,\C)^\C $ with $\sl(2,\C) \oplus \sl(2,\C)$ via the isomorphism above. Under this identification, the complex conjugation of $\sl(2,\C)^\C$ is given by $$\sigma:(X,Y) \mapsto (\overline Y, \overline X). $$
Note that the natural embedding of $\sl(2,\C) \xhookrightarrow{} \sl(2,\C)^\C$ is given by $X \mapsto (X, \overline X)$.

\subsection{Real automorphisms of \texorpdfstring{$\sl(2,\C)$}{sl(2,C)}}
\label{subsection_real_automorphisms}
For each $A \in \SL(2,\C)$, let $\Ad_A:\sl(2,\C) \rightarrow \sl(2,\C)$ be the automorphism $\Ad_A X = AXA^{-1}$, and consider the Lie group homomorphism $\Ad:\SL(2,\C) \rightarrow \Aut_\R(\sl(2,\C))$, $A \mapsto \Ad_A$.
\begin{proposition}
    \label{prop_automorphisms}
    Every real automorphism of $\sl(2,\C)$ is of the form $\Ad_A$, or $\Ad_A$ followed by $X \mapsto \overline X$, where $A \in \SL(2,\C)$. In symbols:
    $$\Aut_\R(\sl(2,\C)) = \{\Id, X \mapsto \overline X\} \cdot \Ad(\SL(2,\C)) .$$
\end{proposition}
\begin{proof}
    By \cite[Proposition 4.1]{Djo99}, we can write
    $$\Aut_\R(\sl(2,\C)) = \{\Id, X \mapsto \overline X\} \cdot \Aut_\C(\sl(2,\C)),$$
    where $\Aut_\C(\sl(2,\C))$ denotes the group of complex-linear automorphisms. By \cite[Proposition D.40]{fulton}, $\Aut_\C(\sl(2,\C))$ is connected because the Dynkin diagram of $\sl(2,\C)$ has no non-trivial automorphisms. Since $\sl(2,\C)$ is a complex semisimple Lie algebra, the identity component of $\Aut_\C(\sl(2,\C))$ is precisely the subgroup of inner automorphisms $\Ad(\SL(2,\C))$ \cite[Proposition 1.98]{knapp}.
\end{proof}

\begin{remark}
    \label{remark_group_automorphism}
    The real automorphism group $\Aut(\SL(2,\C))$ is given by 
    $$\Aut(\SL(2,\C)) = \{\Id, A \mapsto \overline A\} \cdot \left\{C_A : A \in \SL(2,\C)\right\},$$
    where $C_A: \SL(2,\C) \rightarrow \SL(2,\C)$ is given by $C_A B = ABA^{-1}$. Indeed, since $\SL(2,\C)$ is simply connected, the differential $d:\Aut(\SL(2,\C)) \rightarrow  \Aut_\R(\sl(2,\C))$ is an isomorphism. We find 
    $d C_A = \Ad_A$ and $d(A \mapsto \overline A) = (X \mapsto \overline X)$. 

    Observe that every automorphism of $\SL(2,\C)$ preserves $\{\pm I_2\}$, and therefore descends to $\PSL(2,\C) = \SL(2,\C) / \{\pm I_2\}$. Conversely, every automorphism of $\PSL(2,\C)$ lifts to $\SL(2,\C)$.
\end{remark}

Finally, let us determine the image of the $\C$-linear extension map $\Aut_\R(\sl(2,\C)) \xhookrightarrow{} \Aut_\C(\sl(2,\C)^\C),$ where we have identified $\sl(2,\C)^\C \cong \sl(2,\C) \oplus \sl(2,\C)$ (see Section \ref{subsection_complexification}).
\begin{proposition}
    The $\C$-linear extension of $\Ad_A: \sl(2,\C) \rightarrow \sl(2,\C)$ is given by
    $$\theta_A: \sl(2,\C)^\C \longrightarrow \sl(2,\C)^\C, \qquad (X,Y) \longmapsto (AXA^{-1}, \overline A Y \overline A^{-1}),$$
    and the $\C$-linear extension of $\sl(2,\C) \rightarrow \sl(2,\C)$, $X \mapsto \overline X$ is given by 
    $$\flip: \sl(2,\C)^\C \longrightarrow \sl(2,\C)^\C, \qquad (X,Y) \longmapsto (Y,X).$$
\end{proposition}
\begin{proof}
    Identify the real Lie algebra $\sl(2,\C)$ with its image under the embedding $\sl(2,\C) \xhookrightarrow{} \sl(2,\C)^\C$, $X \mapsto (X, \overline X)$. Under this identification, it is easy to see that $\Ad_A = \theta_A|_{\sl(2,\C)}$ and 
    $(X \mapsto \overline X) = \flip|_{\sl(2,\C)}$.
    Since $\theta_A$ and $\flip$ are $\C$-linear, we are done.
\end{proof}

\subsection{Borel subalgebras of \texorpdfstring{$\sl(2,\C)$}{sl(2,C)}} 
\label{subsection_borel}
Recall that a Borel subalgebra of a complex semisimple Lie algebra is a maximal solvable subalgebra. For $\sl(2,\C)$, the Borel subalgebras are precisely the $2$-dimensional subalgebras. Indeed, any $2$-dimensional complex Lie algebra is solvable.
The Borel subalgebras of $\sl(2,\C)$ are parameterised by the complex projective line $\P^1$, i.e. the complex $1$-dimensional subspaces $\ell$ of $\C^2$:

\begin{proposition}
    \label{borel subalgebras parameterisation}
    The map $\Phi:\P^1 \rightarrow \left\{\text{\normalfont Borel subalgebras of $\sl(2,\C)$}\right\}$ given by
    $$ \ell \mapsto \b_\ell := \left\{ X \in \sl(2,\C): X \ell \subseteq \ell\right\}$$
    is a well-defined bijection. Moreover, $\Phi$ is $\SL(2,\C)$-equivariant, where $\SL(2,\C)$ acts on $\P^1$ via $A \cdot \ell = A\ell$, and on the Borel subalgebras via $A \cdot \b = A \b A^{-1}$.
\end{proposition}
\begin{proof}
    Fix $\ell \in \P^1$. It is easy to see that $\b_\ell$ is a complex subalgebra of $\sl(2,\C)$. To see that $\b_\ell$ is $2$-dimensional, fix a non-zero $v \in \ell$ and choose $w \in \C^2 \backslash \ell$, so that $\{v,w\}$ is a basis of $\C^2$. In this ordered basis, we find 
    $$\b_\ell = \left\{ \begin{pmatrix}
        a & b \\ 0 & -a
    \end{pmatrix}: a,b \in \C \right\}.$$

    To show injectivity, suppose $\b_\ell = \b_{\ell'}$ for some $\ell, \ell' \in \P^1$. If $\ell \neq \ell'$, then every $X \in \b_\ell$ has two linearly independent eigenvectors. However, by the previous paragraph, there always exists $X \in \b_\ell$ which does not have two linearly independent eigenvectors (for example, take  the matrix $X$ with $Xv = 0$ and $Xw = v$), a contradiction.

    To show surjectivity, fix a Borel subalgebra $\b \leq \sl(2,\C)$. Since $\b$ is a solvable matrix Lie algebra, Lie's theorem (see \cite[\S4.1]{humphreys} or \cite[Corollary 1.29]{knapp}) implies that there exists $\ell \in \P^1$ such that $X \ell \subseteq \ell$ for all $X \in \b$, so $\b = \b_\ell$.

    Finally, to show that $\Phi$ is $\SL(2,\C)$-equivariant, fix $\ell \in \P^1$, and let us show that $A \b_\ell A^{-1} = \b_{A \ell}$. Indeed, if $X \in \b_\ell$, then $(AXA^{-1} )(A \ell) = AX \ell \subseteq A \ell$, so $A \b_\ell A^{-1} \subseteq  \b_{A \ell}$. The opposite inclusion holds for dimensional reasons.
\end{proof}
Recall that the action of $\SL(2,\C)$ on $\P^1$ by M\"obius transformations is \textit{$2$-transitive}: if $\ell_1 \neq \ell_2$ and  $\ell_1' \neq \ell_2'$ belong to $\P^1$, then there exists $A \in \SL(2,\C)$  such that $A \ell_i = \ell_i'$. Proposition \ref{borel subalgebras parameterisation} immediately tells us the following:

\begin{corollary}
    \label{two point transitivity}
    If $\b_1, \b_2, \b_1', \b_2'$ are Borel subalgebras of $\sl(2,\C)$ with  
    $\b_1 \neq \b_2$ and $\b_1' \neq \b_2'$, then there exists $A \in \SL(2,\C)$ such that $A \b_i A^{-1} = \b_i'$ for $i = 1,2$. In particular, $\SL(2,\C)$ acts transitively on the set of Borel subalgebras.
\end{corollary}

The following proposition is useful:

\begin{proposition}
    \label{every element lies in some Borel subalgebra}
    Let $X,Y \in \sl(2,\C)$.
    \begin{enumerate}[\normalfont (i)]
        \item\label{everyXinBorel} $X$ lies in some Borel subalgebra.
        \item\label{pBor2} If $X$ and $Y$ are linearly independent and $[X,Y] \in \C Y$, then $\det X \neq 0$ and $\det Y = 0$.
    \end{enumerate}
\end{proposition}
\begin{proof}
    To see \eqref{everyXinBorel}, recall that every endomorphism $X:\C^2 \rightarrow \C^2$ has an eigenvector $v$, so $X \in \b_{\C v}$.
    
    Let us show \eqref{pBor2}. Observe that $\C\{X,Y\}$ is a Borel subalgebra because it is a 2-dimensional subalgebra. Since matrix conjugation preserves determinants, and $\SL(2,\C)$ acts transitively on the set of Borel subalgebras of $\sl(2,\C)$, we may assume that 
    $$\C\{X,Y\} = \C\{H,E \}.$$
    Since $\det(aH  + bE) = -a^2$, we are done if we show that $Y$ belongs to $\C E$. Observe that $[H,Y] \in \C Y$ (write $H$ in terms of $X$ and $Y$), so either $Y \in \C H$ or $Y \in \C E$. For the sake of contradiction, suppose $Y = \lambda H$ for some non-zero $\lambda \in \C$, and write $X = a H  + b E$ for $a,b \in \C$ with $b \neq 0$. Then $[X,Y] = -2 b \lambda E$, so $b = 0$, a contradiction.
\end{proof}

\section{Proof of Theorem \ref{main theorem}}
\label{section_main_theorem}
Let $\g$ denote the real Lie algebra $\sl(2,\C)$, and consider the following elements of $\Q(\g)$, i.e. subalgebras $\q$ of $\sl(2,\C)^\C = \sl(2,\C) \oplus \sl(2,\C)$ such that $\q \oplus \sigma \q = \sl(2,\C)^\C$:
\begin{align}
        \label{complex structures list}
        \begin{split}
            \I &:=\sl(2,\C) \oplus 0, \\
        \II_{\lambda} &:= \C\left\{(H,\lambda H), (E,0), (0,F) \right\} \qquad \text{for $ \lambda \in \C \text{ with $|\lambda| \neq 1$}$}, \\
        \III &:= \C\{(H,0), (E,0), (0,H + F)\}.
        \end{split}   
\end{align}

\begin{remark}
    The Lie group $\SL(2,\C)$ has a standard complex structure $J_{\std}$, obtained by viewing $\SL(2,\C)$ as a complex submanifold of $\C^{2 \times 2}$. This complex structure $J_{\std}$ is bi-invariant, and the corresponding $\i$-eigenspace is the subalgebra $\I \in \Q(\g)$.
\end{remark}

Let $\D = \{\lambda \in \C : |\lambda| < 1\}$. Thanks to Proposition \ref{prop_equivalences}, Theorem \ref{main theorem} is  an immediate consequence of Lemma \ref{lem_regularity} and the following:
\begin{theorem}
    \label{theorem classification}
    For the real Lie algebra $\g = \sl(2,\C)$, the $\Aut_\R(\g)$-orbit of each $\q \in \Q(\g)$ contains exactly one of $\I$, $\II_\lambda$ for $\lambda \in \D$, or $\III$.
\end{theorem}

The proof of Theorem \ref{theorem classification} is structured as follows:
in Section \ref{subsection_reduction}, we show that every $\q \in \Q(\g)$ is equivalent to one of the subalgebras listed in (\ref{complex structures list}). In Section \ref{ssecNonEq}, we show that the subalgebras listed in (\ref{complex structures list}) are pairwise non-equivalent.

\subsection{Reduction to normal forms}
\label{subsection_reduction}
Let 
$\pi_1,\pi_2: \sl(2,\C)^\C = \sl(2,\C) \oplus \sl(2,\C) \rightarrow \sl(2,\C)$ be the projection maps onto the first and second summands, respectively. Fix $\q \in \Q(\g)$, and define
\begin{align*}
    d_i :=& \dim_\C \pi_i(\q), \quad k_i := \dim_\C\ker (\pi_i|_\q ) \qquad  \text{ for }i = 1,2.
\end{align*}
By the rank-nullity theorem, we have 
$3 = d_1 + k_1 = d_2 + k_2.$
Moreover, $$\ker (\pi_1|_\q ) \oplus \ker (\pi_2|_\q )  =  (\q \cap \ker \pi_1) \oplus (\q \cap \ker \pi_2)\leq \q,$$ so $k_1 + k_2 \leq 3$. 
It follows that $d_1 + d_2 \geq 3$: indeed, if $d_1 + d_2 \leq 2$, then $k_1 + k_2 \geq 4$, a contradiction.
By applying the automorphism $\flip:(X,Y) \mapsto (Y,X)$, we may assume that $d_1 \geq d_2$. The following proposition covers the remaining cases for $(d_1,d_2)$:
\begin{proposition}\label{pCstr}
    Fix $\q \in \Q(\g)$.
    \begin{enumerate}[\normalfont (i)]
        \item\label{pCstr1} If $d_1 = 3$, then $d_2 = 0$ and $\q=\I$.
        \item\label{pCstr2} If $(d_1,d_2) = (2,2)$, then $\q$ is equivalent to $\II_\lambda$ for $\lambda \in \C$ satisfying $|\lambda| < 1$ and $\lambda \neq 0$.
        \item\label{pCstr3} If $(d_1, d_2) = (2,1)$, then $\q$ is equivalent to either $\II_0$ or $\III$.
    \end{enumerate}
\end{proposition}

\begin{proof}[Proof of Proposition \ref{pCstr}(\ref{pCstr1})]
    In this case, $d_1 = 3$, so $\pi_1|_\q:\q \rightarrow \sl(2,\C)$ is an isomorphism. It follows that 
$$\q = \left\{ (X, \varphi X): X \in \sl(2,\C)\right\},$$
where $\varphi$ is the composition of Lie algebra homomorphisms 
$$\varphi:\sl(2,\C) \xrightarrow{(\pi_1|_\q)^{-1}} \q  \xrightarrow{\pi_2 |_\q} \sl(2,\C).$$
Because $\sl(2,\C)$ is simple, $\varphi$ is either zero, or a complex automorphism of $\sl(2,\C)$. If $\varphi = 0$, then $\q = \sl(2,\C) \oplus 0$ and we are done.

Let us show that the other case cannot happen. For the sake of contradiction, suppose $\varphi:\sl(2,\C) \rightarrow \sl(2,\C)$ is an automorphism. Let us reach a contradiction by showing that $\q \cap \sigma \q \neq 0$. By the proof of Proposition \ref{prop_automorphisms}, there exists a fixed $A \in \SL(2,\C)$ such that $\varphi(X) = AXA^{-1}$ for all $X \in \sl(2,\C)$. Thus, we may write
\begin{align*}
    \q &= \left\{ (X, AXA^{-1}) : X \in \sl(2,\C)\right\}, \\
    \sigma \q& = \left\{ ( \overline A \overline X \overline A^{-1}, \overline X) : X \in \sl(2,\C)\right\} = \left\{ (X, \overline A^{-1} X \overline A) : X \in \sl(2,\C)\right\}.
\end{align*}
Therefore, to show $\q \cap \sigma \q \neq 0$, it suffices to find a non-zero $X \in \sl(2,\C)$ such that $AXA^{-1} = \overline A^{-1} X \overline A$, which is equivalent to $BX = XB$, where $B  = \overline A A$. If $B = \pm I_2$, then any $X \in \sl(2,\C)$ commutes with $B$. If $B \neq \pm I_2$, then $X  = B - \frac{\tr B}{2} I_2$ is non-zero, traceless, and commutes with $B$. 
\end{proof}

\begin{proof}[Proof of Proposition \ref{pCstr}(\ref{pCstr2})] 
    In this case, $(d_1, d_2) = (2,2)$.
    Observe that $\b_1 = \pi_1(\q)$ and $\b_2 = \pi_2(\q)$ are Borel subalgebras. Let us show that $\b_1 \neq \overline{\b_2}$, where $\overline{(\cdot)}$ denotes entrywise matrix conjugation in $\sl(2,\C)$. For the sake of contradiction, suppose $\b_1 = \overline{\b_2}$. Observe that $\overline \b_2 = \overline{\pi_2(\q)} = \pi_1(\sigma \q)$. Therefore, 
$$\sl(2,\C) = \pi_1(\sl(2,\C)^\C) = \pi_1(\q \oplus \sigma \q) = \pi_1(\q) + \pi_1(\sigma \q) =  \pi_1(\q),$$
which contradicts $d_1 = 2$. Now, since $\b_1 \neq \overline{\b_2}$, Corollary \ref{two point transitivity} tells us that there exists $A \in \SL(2,\C)$ such that 
$$A \b_1 A^{-1} = \C\{H,E\}, \qquad A \overline{\b_2}A^{-1} = \C \{H,F\}.$$
Applying $\overline{(\cdot)}$ to the second equation gives $\overline A \b_2 \overline A^{-1} = \C\{H,F\}$. Thus, by replacing $\q$ with $\theta_A\q$, we assume that $\pi_1(\q) = \C\{H,E\}$ and $\pi_2(\q) =  \C\{H,F\}$. Since $k_1 = k_2 = 1$, there exist $X,Y,Z,W \in \sl(2,\C)$ such that 
$$\q = \C\left\{(X,Y), (Z,0), (0,W) \right\},$$
where $X,Z \in \C\{H,E\}$ are linearly independent, and $Y,W \in \C\{H,F\}$ are linearly independent. Since $\q$ is a subalgebra, $[(X,Y),(Z,0)]  = ([X,Z],0)$ must belong to $\q$. Thus, $[X,Z] \in \C Z$. By Proposition \ref{every element lies in some Borel subalgebra} \eqref{pBor2}, it follows that $\det(X) \neq 0$ and $\det(Z) = 0$. The  elements of $\C\{H,E\}$ with determinant zero are precisely the multiples of $E$. Therefore, by rescaling $Z$, we may assume that $Z = E$. By the same reasoning, we may assume that  $W = F$. By subtracting suitable multiples of $(E,0)$ and $(0,F)$ from $(X,Y)$ and rescaling $(X,Y)$, we may assume that $(X,Y) = (H, \lambda H)$ for some $\lambda \in \C$ with $\lambda \neq 0$. In summary, we have shown that, up to equivalence,
$$\q =\C\left\{(H, \lambda H), (E,0), (0,F) \right\} = \II_\lambda. $$

Let us show that $|\lambda| \neq 1$. Since $\q \cap \sigma \q = 0$, the vectors $(H, \lambda H)$ and $( \overline \lambda H, H) = \sigma (H, \lambda H)$ are linearly independent. In particular, 
$$\det\begin{pmatrix}
    1 & \lambda \\ \overline \lambda & 1
\end{pmatrix} = 1 - |\lambda|^2$$
is non-zero.  

We  now prove that the complex structures $\II_\lambda$ and $\II_{1/\lambda}$ are equivalent for any $\lambda \in \C$ with $\lambda \neq 0$ and $|\lambda | \neq 1$. 
    Consider the automorphisms $\flip:(X,Y) \mapsto (Y,X)$ and $\theta_A$, where $A = \begin{pmatrix}
    0 & 1 \\ -1 & 0
\end{pmatrix}$. Observe that $AHA^{-1} = -H$, $AEA^{-1} = -F$, $AFA^{-1} = -E$. Thus, we find 
    \begin{align*}
        \flip (\II_\lambda) &= \C\left\{(H, \textstyle\frac{1}{\lambda} H), (0,E), (F,0) \right\}, \\
        \theta_A \flip (\II_\lambda) &=  \C\left\{(H, \textstyle\frac{1}{\lambda} H), (E,0), (0,F) \right\} = \II_{1/\lambda}.
    \end{align*}
    This completes the proof.
\end{proof}

\begin{proof}[Proof of  Proposition \ref{pCstr}(\ref{pCstr3})]
    In this case, $(d_1,d_2) = (2,1)$, so there exists a Borel subalgebra $\b_1$ of $\sl(2,\C)$ and a non-zero $X \in \sl(2,\C)$ such that 
$$\q = \b_1 \oplus \C X.$$
Since $\sigma \q = \C \overline X \oplus \overline{\b_1}$, and $\q \cap \sigma \q = 0$, it follows that $\overline X$ does not lie in $\b_1$.
By Proposition \ref{every element lies in some Borel subalgebra} \eqref{everyXinBorel}, $X$ belongs to some Borel subalgebra $\b_2$. Since $\b_1 \neq \overline{\b_2}$, by Corollary \ref{two point transitivity}, there exists $A \in \SL(2,\C)$ such that 
$$A \b_1 A^{-1} = \C\{H,E\}, \qquad A \overline{\b_2}A^{-1} = \C \{H,F\}.$$
Applying $\overline{(\cdot)}$ to the second equation gives $\overline A \b_2 \overline A^{-1} = \C\{H,F\}$. Thus, by replacing $\q$ with $\theta_A \q$, we may assume that 
$$\q = \C\{H,E\} \oplus \C X,$$
where $X \in \C\{H,F\}$.
Now, if $X \in \C F$, then we may assume that $X =F$ by rescaling $X$. In this case, $\q = \II_0$.
Finally, suppose $X$ does not lie in $\C F$. Then, after rescaling $X$, we may assume that $X = H + z F$ for some $z \in \C$. Note that $z \neq 0$: if $z = 0$, then $(H,0) \in \q \cap \sigma \q$, which contradicts $\q \cap \sigma \q = 0$. Fix $w \in \C$ with $\overline w^2 = z$, and consider $A = \begin{pmatrix}
    w & 0 \\ 0 & w^{-1}
\end{pmatrix} \in \SL(2,\C)$. We find that
$$AHA^{-1} = H, \qquad AEA^{-1} =  \overline zE, \qquad  \overline A F \overline A^{-1} = z^{-1} F.$$
In particular, conjugation by $A$ preserves $\C\{H,E\}$. Therefore, 
$$\theta_A \q = \C\{H,E\} \oplus \C(H + F) = \III,$$
as desired.
\end{proof}

\subsection{Non-equivalence of the normal forms}
\label{ssecNonEq}
To finish the proof of Theorem \ref{theorem classification}, it remains to show the following:

\begin{proposition}
    \label{prop_non_equivalence}
The subalgebras $\I$, $\II_\lambda$ for $\lambda \in \D$, and $\III$ defined in (\ref{complex structures list}) are pairwise non-equivalent.
\end{proposition}

Before proving Proposition \ref{prop_non_equivalence}, let us determine the regularity of the complex structures listed in (\ref{complex structures list}):

\begin{lemma}
\label{lem_regularity}
The complex structures corresponding to $\I$ and $\II_\lambda$ are regular. The complex structure corresponding to $\III$ is not regular.
\end{lemma}
\begin{proof}
    First, $\I \in \Q(\g)$ is regular because the corresponding complex structure $J_{\std}$ is bi-invariant. Next, $\h = \R\{H, \i H\}$ is a (real) Cartan subalgebra of $\g = \sl(2,\C)$ with $\h^\C = \C H \oplus \C H$. It is easy to see that $[\h^\C, \II_\lambda] \subseteq \II_\lambda$, so $\II_\lambda \in \Q(\g)$ is regular.

    Finally, to show that $\III$ is not regular, it suffices to show that its normaliser in $\sl(2,\C)^\C$ is $N(\III) = \III$. Indeed, if this is the case, then $N(\III)$ cannot contain any non-zero real vectors. To this end, observe that 
    $\III = \b \oplus \C Z$, where $\b = \C\{H,E\}$ is a Borel subalgebra and $Z = H + F$. The line $\C Z$ is a (complex) Cartan subalgebra of $\sl(2,\C)$ because $\det Z \neq 0$. Thus,
    $$N(\III) = N_{\sl(2,\C)} (\b) \oplus N_{\sl(2,\C)} (\C Z) = \b \oplus \C Z = \III,$$
    where the second equality follows because Borel subalgebras are self-normalising by \cite[Chapter 16.3 Lemma A]{humphreys}, and Cartan subalgebras are self-normalising (by definition).
\end{proof}

\begin{proof}[Proof of Proposition \ref{prop_non_equivalence}]
    First, among the complex structures listed in  (\ref{complex structures list}), $\I$ is the only bi-invariant complex structure, and $\III$ is the only non-regular complex structure. These two  properties are preserved by equivalence. 
    Thus, it remains to show that $\II_\lambda$ and $\II_{\lambda'}$ are not equivalent for distinct $\lambda, \lambda' \in \D$; it suffices to show that they are not isomorphic as abstract complex Lie algebras. We are done if we 
    show that $\lambda \in \D$ can be recovered from the abstract Lie algebra $\II_\lambda$.
    
 For $\q = \II_\lambda$, let $e_1 = (H, \lambda H)$, $e_2 = (E,0)$, and $e_3 = (0,F)$. For $X \in \q$, we can write $X = a e_1 + b e_2 + c e_3$, where $a,b,c \in \C$. In the ordered basis $(e_1,e_2,e_3)$, we have 
    $$\rest{\ad_X}{\q} = \begin{pmatrix}
        0 & 0 &0 \\ -2b & 2a & 0 \\ 2c & 0 & - 2a \lambda
    \end{pmatrix}. $$
 In particular, $\rest{\ad_X}{\q}$  has eigenvalues $0, 2a, -2a \lambda$. It follows that $\lambda \in \D$ can be recovered from the  abstract Lie algebra $\II_\lambda$.
\end{proof}

\section{Topology of \texorpdfstring{$\J(\g) / \Aut(\G)$}{J(g)/Aut(G)}}
\label{section_topology}
In this section, we determine the topology of left-invariant complex structures on the real Lie group $G = \SL(2,\C)$ modulo automorphisms, i.e. $\J(\g)/ \Aut(\G)$. We endow $\J(\g) \subseteq \End(\g)$ with the subspace topology, and $\J(\g) / \Aut(\G)$ with the quotient topology.
Equipping $\Q(\g) \subseteq \Gr_n(\g^\C)$ with the subspace topology, the bijection $\J(\g) \cong \Q(\g)$ given in Proposition \ref{prop_bijection_J_Q} is a homeomorphism. Thus, by Proposition \ref{prop_equivalences}, as topological spaces, we can write
$$\J(\g) / \Aut(\G) = \J(\g)/ \Aut_\R(\g) \cong \Q(\g) / \Aut_\R(\g),$$
where the sets above are equipped with the quotient topology.  By Theorem \ref{theorem classification}, we have
$$\M := \Q(\g) / \Aut_\R(\g) = \{[\I]\} \sqcup \{[\II_\lambda]\}_{\lambda \in \D} \sqcup \{[\III]\},$$
where $[\q]$ denotes the $\Aut_\R(\g)$-orbit of $\q \in \Q(\g)$. To simplify notation, we identify the disk $\D \cong \{[\II_\lambda]\}_{_{\lambda \in \D}}$ via $\lambda \leftrightarrow [\II_\lambda]$.

\begin{theorem}
    \label{theorem topology}
    For the real Lie algebra $\g = \sl(2,\C)$, a subset $U \subseteq \Q(\g) \, / \Aut_\R(\g)$ is open in the quotient topology if and only if the following two properties hold:
    \begin{itemize}
        \item\label{top1} $U \cap \D$ is open in the usual topology of $\D$.
        \item\label{top2} If $[\II_0] \in U$, then $[\III] \in U$. 
    \end{itemize}
\end{theorem}

Before proving the theorem, we note some immediate consequences:

\begin{corollary}
    Let $\g$ denote the real Lie algebra $\sl(2,\C)$. 
    \begin{itemize}
        \item A left-invariant complex structure $J \in \J(\g)$ is regular if and only if the orbit of $J$ under $\Aut_\R(\g)$ is closed in $\J(\g)$.
        \item In $\J(\g)$, there is no continuous path from a bi-invariant complex  structure to one which is not bi-invariant.
        \item The orbit $[\III]$ is open in $\Q(\g)$, and its boundary is $[\II_0]$. 
    \end{itemize}
\end{corollary}

% Let $\rho:\Q(\g) \rightarrow  \M = \Q(\g) / \Aut_\R(\g)$ denote the natural projection: this is an open quotient map because the action of $\Aut_\R(\g)$ on $\Q(\g)$ is continuous.
The rest of this section is devoted to the proof of \Cref{theorem topology}.
Let  $\rho:\Q(\g) \rightarrow  \M = \Q(\g) / \Aut_\R(\g)$ denote the natural projection: this is an open quotient map because the action of $\Aut_\R(\g)$ on $\Q(\g)$ is continuous.
Let $\T_\D$ denote the usual topology on the disk, let $ \T_{\quo}$ denote the quotient topology on $\M$,  and consider the following topology on $\M$:
$$\T_{\cand} :=  \left\{ U \subseteq \M : U \cap \D \text{ is open in $\T_\D$, and if $0\in U$ then $[\III] \in U$}\right\}.$$
Theorem \ref{theorem topology} follows if we show that $\T_{\quo} = \T_{\cand}$.

\begin{remark}
    \label{rem_basis}
    A basis for $\T_\cand$ is given by the following sets:
    $$\{[\I]\}, \quad \{[\III]\}, \quad \text{$U \in \T_\D$ with $0 \notin U$}, \quad \text{$U \sqcup \{[\III]\}$ with $U \in \T_\D$ and $0 \in U$}.$$
\end{remark}

\begin{proposition}
    \label{prop topology}
    We have the following facts about the quotient topology $\T_\quo$:
    \begin{enumerate}[\normalfont (i)]
        \item\label{ptop1} The inclusion $\iota_\D: (\D, \T_\D) \xhookrightarrow{} (\M, \T_{\quo})$ is continuous.
        \item\label{ptop2} Every neighbourhood $U \in \T_\quo$ of $0 \in \D$ contains $[\III]$.
        \item\label{ptop3} The singleton $\{[\I]\}$ is both open and closed in $\T_\quo$.
        \item\label{ptop4} The singleton $\{[\III]\}$ is open in $\T_\quo$.
        \item\label{ptop5} The following map is continuous:
        $$\Lambda: (\D \sqcup \{[\III]\}, \T_\quo) \rightarrow (\D, \T_\D), \qquad \begin{cases}
            \lambda \mapsto \lambda & \forall \lambda \in \D, \\
            [\III] \mapsto 0.&
        \end{cases}$$
    \end{enumerate}
\end{proposition}

\begin{remark}
    In \eqref{ptop5}, $\D \sqcup \{[\III]\}$ is equipped with the subspace topology induced by $\T_\quo$. Since $\{[\I]\}$ is closed in $\T_\quo$ by \eqref{ptop3}, this subspace topology  consists precisely of the elements of $\T_\quo$ contained in $\D \sqcup \{[\III]\}$.
\end{remark}

\begin{proof}[Proof of Theorem \ref{theorem topology}, assuming Proposition \ref{prop topology}]
    To show that $\T_\quo \subseteq \T_\cand$, suppose $U \in \T_\quo$. Property \eqref{ptop1} implies that $U \cap \D = \iota_\D^{-1}(U) \in \T_\D$. By  \eqref{ptop2}, $0 \in U$ implies $[\III] \in U$. Thus, $U \in \T_\cand$.

    To show that $\T_\cand \subseteq \T_\quo$, it suffices to show that $\T_\quo$ contains the basis elements of $\T_\cand$ listed in Remark \ref{rem_basis}. First, the singletons $\{[\I]\}, \{[\III]\}$ are open in $\T_\quo$ by \eqref{ptop3} and \eqref{ptop4}, respectively. Next, let $U \in \T_\D$. By \eqref{ptop5}, the following hold:  $0 \notin U$ implies $U=\Lambda^{-1}(U)  \in \T_{\quo}$, and $0 \in U$ implies $  U \sqcup \{[\III]\}  = \Lambda^{-1}(U)\in \T_{\quo}$.
\end{proof}

It remains to prove Proposition \ref{prop topology}. Henceforth, we assume that $\M$ is equipped with the quotient topology $\T_\quo$. 

First, let us introduce some notation: for an $n$-dimensional complex vector space $V$ and a positive integer $k \leq n$, let $\Gr_k(V)$ denote the Grassmannian of $k$-dimensional complex subspaces of $V$. Recall that the \textit{tautological bundle} $p:\CS \rightarrow \Gr_k(V)$ is the rank $k$ subbundle of the trivial bundle $\Gr_k(V) \times V$ defined by 
$$\CS := \left\{ (\q,X) \in \Gr_k(V) \times V : X \in \q \right\}.$$
Next, consider the \textit{Stiefel bundle} $q: \St_k(V) \rightarrow \Gr_k(V)$, where
\begin{align*}
    \St_k(V) :=& \left\{\text{Injective linear maps $A: \C^k \xhookrightarrow{} V$} \right\} \\
    =& \left\{ (v_1,\ldots,v_k) \in V^{\oplus k} : \text{$v_1,\ldots, v_k$ are linearly independent}\right\},
\end{align*}
and $q:\St_k(V) \rightarrow \Gr_k(V)$ is given by $A \mapsto \Im (A)$, the image of $A$. The total space $\St_k(V)$ is an open subset of $\Hom(\C^k, V)$ by \cite[Proposition B.25]{leesm}.

\begin{lemma}
    \label{lemma  grassmannian}
    Let $T:V \rightarrow W$ be a $\C$-linear map between finite-dimensional $\C$-vector spaces, and let $1 \leq d \leq k \leq \dim V$ be integers.
    \begin{enumerate}[\normalfont (a)]
        \item\label{lGr1} The subset $\mathcal U := \left\{ \q \in \Gr_k(V) : \rank(T|_\q) \geq d\right\}$ is open in $\Gr_k(V)$.
        \item\label{lGr2} Let $\mathcal D := \left\{ \q \in \Gr_k(V) : \rank(T|_\q) = d\right\}$. The following map is continuous:
        $$\eta: \mathcal D \rightarrow \Gr_{k-d}(\ker T), \qquad \q \mapsto \q \cap \ker T = \ker(T|_\q).$$
    \end{enumerate}
\end{lemma}
\begin{proof}[Proof of Lemma \ref{lemma  grassmannian}]
    To show \eqref{lGr1}, it suffices to show that $q^{-1}(\mathcal U)$ is open in $\St_k(V)$ because $q:\St_k(V) \rightarrow \Gr_k(V)$ is an open quotient map. Observe that 
    \begin{align*}
        q^{-1}(\mathcal U) &= \left\{ A \in \St_k(V) : \rank(T \circ A) \geq d\right\} \\
        &= (L_T)^{-1} \left(\{B \in \Hom(\C^k, W) : \rank B \geq d\} \right),
    \end{align*}
    where $L_T: \St_k(V) \rightarrow \Hom(\C^k, W)$ is the continuous map $A \mapsto T \circ A$. The set $\{\rank B \geq d\}$ is open in $\Hom(\C^k, W)$ by \cite[Proposition B.25]{leesm}, so we are done.

    To show \eqref{lGr2}, observe that we have a continuous bundle homomorphism over $\mathcal D$
     $$\Psi:\CS|_{\mathcal D} \hookrightarrow \mathcal D \times V \xrightarrow{\displaystyle (\q,X) \mapsto (\q, TX)}\mathcal D \times W,$$
    %\begin{align*}
    %\Psi:\CS|_{\mathcal D} \hookrightarrow \mathcal D \times V&\rightarrow\mathcal D \times W,\\
    %(\q,X)& \mapsto (\q, TX),
    %\end{align*} 
    where the fibre maps are $\Psi_\q = T|_{\q} : p^{-1}(\q)=\q \rightarrow W$. By definition of $\mathcal D$, $\Psi$ has constant rank $d$, so $\ker \Psi$ is a rank $(k-d)$ subbundle of the trivial bundle $\mathcal D \times \ker T \rightarrow \mathcal D$, with fibres $(\ker \Psi)_\q = \q \cap \ker T$. To show that $\eta$ is continuous, let $s_1,\ldots, s_{k-d}:\mathcal O \subseteq \mathcal D \rightarrow  \ker \Psi \subseteq \mathcal D \times \ker T$ be a continuous local frame for $\ker \Psi$. Then $\eta|_{\mathcal O}$ is precisely the composition of the following continuous maps:
     $$\mathcal O \xrightarrow{\displaystyle \q \mapsto (s_1(\q),\ldots, s_{k-d}(\q))} \St_{k-d}(\ker T) \twoheadrightarrow \Gr_{k-d}(\ker T).$$
    %\begin{align*}
    %\mathcal O \rightarrow \St_{k-d}(\ker T) \twoheadrightarrow \Gr_{k-d}(\ker T),\\
    %\q &\mapsto (s_1(\q),\ldots, s_{k-d}(\q)).
    %\end{align*}
    This completes the proof.
\end{proof}

\begin{lemma}
    \label{lemma for v}
    For each $\q \in \Q(\g)$ and $X \in \q$, set $\ad^\q_X := \rest{\ad_X}{\q}: \q \rightarrow \q$.
    \begin{enumerate}[\normalfont (a)]
        \item\label{lv1} Let $\End(\CS|_{\Q(\g)}) \rightarrow \Q(\g)$ denote the endomorphism bundle of the tautological bundle $\CS|_{\Q(\g)} \rightarrow \Q(\g)$. The following map is continuous:
        $$\Phi: \CS|_{\Q(\g)} \rightarrow \End(\CS|_{\Q(\g)}), \qquad (\q,X) \mapsto \ad^\q_X : \q \rightarrow \q.$$
        \item\label{lv2} The maps $\tau, \delta: \CS|_{\Q(\g)} \rightarrow \C$ given by
        $$\tau(\q,X) := \tr \ad^\q_X, \qquad \delta(\q,X) := \frac12 \left( \tau(\q,X)^2 - \tr ((\ad^\q_X)^2)\right)$$
        are continuous and invariant under the action of $\Aut_\R(\g)$ on $\CS|_{\Q(\g)}$ given by $\varphi \cdot (\q,X) = (\varphi \q, \varphi X)$.
        \item\label{lv3} The set $\mathcal O = \left\{(\q, X) \in \CS|_{\Q(\g)} : \tau(\q,X) \neq 0\right\}$ is open, and the following map is continuous:
        $$R:\mathcal O \rightarrow \C, \qquad  R(\q,X) := \frac{\delta(\q,X)}{\tau(\q,X)^2}.$$
        \item\label{lv4} Fix $(\q, X) \in \mathcal O$. Then $\q \in \Q(\g) \backslash [\I]$, and 
        $$R(\q,X) = \begin{cases}
            -\lambda / (1- \lambda)^2
             & \text{if $\q \in [\II_\lambda]$, } \\
            0 & \text{if $\q \in [\III]$.}
        \end{cases}$$
        \item\label{lv5} The map $K: \D \rightarrow \C \backslash [1/4, \infty)$ given by $\lambda \mapsto -\lambda / (1- \lambda)^2$ is a biholomorphism.
    \end{enumerate}
\end{lemma}
\begin{remark}
    % For $(\q,X) \in \CS|_{\Q(\g)}$, the characteristic polynomial of $\ad^\q_X:\q \rightarrow \q$ is given by $\det(t \Id_{\q} - \ad^\q_X) = t^3 - \tau(\q,X) t^2 + \delta(\q,X) t$.
    For $(\q,X) \in \CS|_{\Q(\g)}$, the quantities $\tau(\q,X),\delta(\q,X)$ can be recovered from the characteristic polynomial of $\ad^\q_X:\q \rightarrow \q$, which is 
    $$\det(t \Id_{\q} - \ad^\q_X) = t^3 - \tau(\q,X) t^2 + \delta(\q,X) t.$$
\end{remark}

\begin{proof}[Proof of Lemma \ref{lemma for v}]
    To show \eqref{lv1}, it suffices to show that $\Phi(s_1)s_2: \mathcal U \subseteq \Q(\g) \rightarrow \CS|_{\Q(\g)}$ is continuous for any continuous local sections $s_1,s_2: \mathcal U \rightarrow \CS|_{\Q(\g)}$. Let $X_1,X_2:\mathcal U \rightarrow \g^\C$ be the continuous maps such that $s_i(\q) = (\q, X_i(\q))$. Then $\Phi(s_1)s_2$ is equal to the following composition of continuous maps:
    $$\mathcal U \xrightarrow{\displaystyle \q \mapsto (\q, X_1(\q), X_2(\q))} \mathcal U \times \g^\C \times \g^\C \xrightarrow{\displaystyle (\q,X,Y) \mapsto (\q, [X,Y])} \mathcal U \times \g^\C.$$

    Let us show \eqref{lv2}. First, the continuity of $\tau$ and $\delta$ follows from the continuity of the trace $\tr:\End(\CS|_{\Q(\g)}) \rightarrow \C$ and square $(\cdot)^2:\End(\CS|_{\Q(\g)}) \rightarrow \End(\CS|_{\Q(\g)})$. The $\Aut_\R(\g)$-invariance of $\tau$ and $\delta$ follows from the equality
    $$\ad^{\varphi \q}_{\varphi X} = \varphi \circ \ad^{\q}_X \circ \varphi^{-1}$$
    for all $(\q,X) \in \CS|_{\Q(\g)}$ and $\varphi \in \Aut_\R(\g)$. Part \eqref{lv3} follows immediately from \eqref{lv2}.

    Let us show \eqref{lv4}. First, if $\q \in [\I]$, then $\q \cong \sl(2,\C)$ is unimodular, so $\tau(\q,X) = 0$ for all $X \in \q$. By $\Aut_\R(\g)$-invariance, it suffices to compute $R(\q,X)$ for $\q = \II_\lambda$ and $\q = \III$. For $\q = \II_\lambda$, let $e_1 = (H, \lambda H)$, $e_2 = (E,0)$, and $e_3 = (0,F)$. Write $X \in \q$ as $X = a e_1 + b e_2 + c e_3$, where $a,b,c \in \C$. In the ordered basis $(e_1,e_2,e_3)$, we have 
    $$\ad^{\II_\lambda}_X = \begin{pmatrix}
        0 & 0 &0 \\ -2b & 2a & 0 \\ 2c & 0 & - 2a \lambda
    \end{pmatrix}, \qquad \left(\ad^{\II_\lambda}_X\right)^2 = \begin{pmatrix}
        0 & 0 &0 \\ -4ab & 4a^2 & 0 \\ -4ac \lambda & 0 & 4 a^2 \lambda^2
    \end{pmatrix}.$$
    Thus, 
    $$R(\II_\lambda,X) = \frac{\delta(\II_\lambda,X)}{\tau(\II_\lambda,X)^2} =  \frac{\frac12 (4a^2 (1-\lambda)^2 - 4a^2 (1 + \lambda^2))}{4a^2 (1- \lambda)^2} = \frac{-\lambda}{(1-\lambda)^2}.$$
    The computation for $R(\III,X)$ is analogous.

    Finally, \eqref{lv5} follows from \cite[Chapter 7-4]{ahlfors}.
\end{proof}

\begin{proof}[Proof of Proposition \ref{prop topology}]
% \begin{proof}[Proof of Proposition \ref{prop topology}, assuming Lemmas \ref{lemma  grassmannian} and \ref{lemma for v}]
    For the real Lie algebra $\g = \sl(2,\C)$, consider $\Q(\g) \subseteq \Gr_3(\g^\C)$. Let $$p:\CS|_{\Q(\g)} \rightarrow \Q(\g) \qquad \text{and} \qquad  q:\St_3(\g^\C)|_{\Q(\g)} \rightarrow \Q(\g)$$ denote the restrictions of the tautological bundle and the Stiefel bundle of $\Gr_3(\g^\C)$ to $\Q(\g)$.

To show \eqref{ptop1}, observe that $\iota_\D$ is a composition of continuous maps:
 $$\iota_\D: \D \xhookrightarrow{\displaystyle \lambda \mapsto ((H, \lambda H), (E,0), (0,F))} \St_3(\g^\C)|_{\Q(\g)} \overset{q}{\twoheadrightarrow} \Q(\g) \overset{\rho}{\twoheadrightarrow} \M.$$
%\begin{align*}
%\iota_\D: \D &\xhookrightarrow{\hspace{1.5cm}} %\St_3(\g^\C)|_{\Q(\g)} \overset{q}{\twoheadrightarrow} \Q(\g) \overset{\rho}{\twoheadrightarrow} \M\\*
%\lambda &\mapsto ((H, \lambda H), (E,0), (0,F))
%\end{align*}

To show \eqref{ptop2}, consider the curve $\gamma:\R \rightarrow \M$ defined by the following composition of continuous maps:
$$\gamma: \R \xrightarrow{\displaystyle t \mapsto ((H,0), (E,0), (0,tH + F))}  \St_3(\g^\C)|_{\Q(\g)}\overset{q}{\twoheadrightarrow} \Q(\g) \overset{\rho}{\twoheadrightarrow} \M.$$
%\begin{align*}
%\gamma: \R &\xhookrightarrow{\hspace{2cm}} \St_3(\g^\C)|_{\Q(\g)} \overset{q}{\twoheadrightarrow} \Q(\g) \overset{\rho}{\twoheadrightarrow} \M\\*
%t &\mapsto ((H,0), (E,0), (0,tH + F))
%\end{align*}
Observe that $\det(t H + F) = 0$ if and only if $t = 0$. Thus, 
the argument in the proof of Proposition \ref{pCstr} \eqref{pCstr3} shows that $\gamma(0) = [\II_0] \in \M$ and $\gamma(t) = [\III]$ for $t \neq 0$. Therefore, every neighbourhood of $[\II_0]$ in $\M$ contains $[\III]$.

Let us show \eqref{ptop3}. The argument in the proof of Proposition \ref{pCstr} \eqref{pCstr1} implies that $\rho^{-1}([\I]) = \{\I, \flip (\I)\}$, where $\flip:\g^\C \rightarrow \g^\C$ is the automorphism  $(X,Y) \mapsto (Y,X)$. Thus, $\{[\I]\}$ is closed. Moreover, 
$$\{\I \} = \left\{ \q \in \Q(\g) : \rank(\pi_1|_{\q}) \geq 3\right\}, \qquad \{\flip (\I) \} = \left\{ \q \in \Q(\g) : \rank(\pi_2|_{\q}) \geq 3\right\}.$$
Thus, Lemma \ref{lemma  grassmannian} \eqref{lGr1} implies that $\{\I\}$ and $\{\flip (\I) \}$ are open in $\Q(\g)$, so $\{[\I]\}$ is also open.

Let us show \eqref{ptop4}. Because $\rho: \Q(\g) \rightarrow \M$ is an open quotient map, it suffices to find a non-empty open subset $\mathcal U$ of $\Q(\g)$ such that $\mathcal U \subseteq [\III]$. By \eqref{ptop3} and Lemma \ref{lemma  grassmannian} \eqref{lGr1}, the following set is open in $\Q(\g)$:
$$\mathcal D := \left\{\q \in \Q(\g) : \rank(\pi_1|_\q) = 2 \right\} = \left\{\q \in \Q(\g) : \rank(\pi_1|_\q) \geq  2 \right\} \backslash \{\I\}.$$
Now, recall that $\ker \pi_1 = 0 \oplus \sl(2,\C)$. By Lemma \ref{lemma  grassmannian} \eqref{lGr2}, the following map is continuous:
$$\eta: \mathcal D \rightarrow \P(\ker \pi_1), \qquad \q \mapsto \q \cap \ker \pi_1 = \ker(\pi_1|_\q).$$
The arguments in the proofs of Proposition \ref{pCstr}, parts \eqref{pCstr2} and \eqref{pCstr3}, show that 
$$\mathcal U := \eta^{-1}\left(\left\{ \C(0,Z) \in \P(\ker \pi_1) : \det Z \neq 0\right\}\right)$$
is a subset of the orbit $[\III]$. Moreover, $\mathcal U$ is open, because $\{(0,Z)  : \det Z \neq 0\}$ is open in $\ker \pi_1$.

Finally, let us show \eqref{ptop5}. Since $\rho: \Q(\g) \rightarrow \M$ is a quotient map, it suffices to show that $\widetilde \Lambda:= \Lambda \circ \rho: \Q(\g) \backslash [\I] \rightarrow \D$ is continuous. Fix $\q_0 \in \Q(\g) \backslash [\I]$, and let $s: \mathcal U \subseteq \Q(\g) \backslash[\I] \rightarrow \CS|_{\Q(\g)}$ be a continuous local section such that $s(\q_0) \in \mathcal O$ (this can be done because $\q_0$ is not unimodular). By replacing $\mathcal U$ with $s^{-1}(\mathcal O)$ if necessary, we may assume that the image of $s$ is contained in $\mathcal O$.  By Lemma \ref{lemma for v},   $\widetilde \Lambda|_{\mathcal U}$ can be written as a composition of continuous maps:
$$\mathcal U \xrightarrow{s} \mathcal O \xrightarrow{R} \C \backslash [1/4, \infty) \xrightarrow{K^{-1}} \D, $$
where $R$ and $K$ are defined in  Lemma \ref{lemma for v}.
\end{proof}

\section{Balanced and pluriclosed metrics}
\label{section_balanced_pluriclosed}
In this section, we prove the following theorem:
\begin{theorem}
    \label{theorem_balanced_pluriclosed}
    Let $\Gamma$ be a cocompact lattice of $\SL(2,\C)$ and let $J$ be a left-invariant complex structure. Then $(\Gamma \backslash \SL(2,\C), J)$ admits a balanced metric, but does not admit any pluriclosed metrics.
\end{theorem}

By the following well-known result, it suffices to consider \textit{left-invariant} balanced and pluriclosed metrics on $(\SL(2,\C),J)$ (see \cite[Theorem 2.1]{FG} and \cite[Proposition 3.6]{ugarte}):
% (see \cite[Section 2.2]{kwong} for a proof):
\begin{proposition}
    \label{proposition_symmetrisation}
    Let $\G$ be a real Lie group, let $\Gamma \leq \G$ be a cocompact lattice, and let $J$ be a left-invariant complex structure on $\G$. The following are equivalent:
    \begin{enumerate}[\normalfont (i)]
        \item $(\Gamma \backslash \G,J)$ admits a balanced/pluriclosed metric.
        \item $(\G,J)$ admits a left-invariant balanced/pluriclosed metric.
    \end{enumerate}
\end{proposition}

The following lemma provides an obstruction to the existence of pluriclosed metrics:

\begin{lemma}
    \label{lem_pluriclosed_obstruction}
    Let $\G^{2n}$ be a unimodular Lie group, let $J$ be a left-invariant complex structure on $\G$, and let $\varepsilon^1,\ldots, \varepsilon^n$ be a basis for $\Lambda^{1,0} \g^*$. Suppose there exists $\eta \in \Lambda^{n-2,n-2} \g^*$ such that 
    $$\partial \overline \partial \eta = \varepsilon^1 \wedge \overline{\varepsilon^1} \wedge \cdots \wedge \varepsilon^{n-1} \wedge \overline{\varepsilon^{n-1}}.$$
    Then $(\G,J)$ does not admit any left-invariant pluriclosed metrics.
\end{lemma}
\begin{proof}
    First, let us show that unimodularity implies that every left-invariant exact top form is zero, i.e. $d(\Lambda^{2n-1} \g^*) = 0$. By unimodularity, there exists a non-zero bi-invariant top form $\mu \in \Lambda^{2n} \g^*$. Since $\mu$ is non-zero, the map $\g \rightarrow \Lambda{2n-1}\g^*$ given by $X \mapsto \iota_X \mu$ is a linear isomorphism, where $\iota$ denotes the interior product. Thus, every element of $\Lambda^{2n-1} \g^*$ is of the form $\iota_X \mu$ for some $X \in \g$. We find 
    $$d(\iota_X \mu) = \mathcal L_X \mu = \frac{d}{dt}\bigg|_{t = 0} (R_{\exp(tX)})^* \mu = 0,$$
    where $\mathcal L$ is the Lie derivative, and $R_{g}:\G \rightarrow \G$ denotes right multiplication by $g \in \G$.

    Now, for the sake of contradiction, suppose that $\omega \in \Lambda^{1,1} \g^*$ is the fundamental form of a left-invariant pluriclosed metric on $(\G,J)$. Write $\omega = \i g_{i\overline j} \; \varepsilon^i \wedge \overline{\varepsilon^j}$. Observe that
    \begin{align*}
        0 = d\left( \overline \partial \eta \wedge \omega + \eta \wedge \partial \omega \right) &= \partial \left(  \overline \partial \eta \wedge \omega\right) + \underbrace{\overline \partial \left(  \overline \partial \eta \wedge \omega\right)}_{= 0 \text{ by type}} + 
        \underbrace{\partial \left( \eta \wedge \partial \omega\right)}_{ = 0 \text{ by type}}+\overline \partial \left( \eta \wedge \partial \omega\right) \\
        &= \partial \overline \partial \eta \wedge \omega - \overline \partial \eta \wedge \partial \omega + \overline \partial \eta \wedge \partial \omega + \eta \wedge \underbrace{ \overline \partial \partial \omega }_{= 0} \\
        &= \i g_{n \overline n} \;\varepsilon^1 \wedge \overline{\varepsilon^1} \wedge \cdots \wedge \varepsilon^{n} \wedge \overline{\varepsilon^{n}} \neq 0,
    \end{align*}
    a contradiction.
\end{proof}

\begin{proof}[Proof of Theorem \ref{theorem_balanced_pluriclosed}]Let $\q$ denote the $\i$-eigenspace of $J$. By Theorem \ref{theorem classification}, it suffices to consider the cases when $\q$ is equal to $\I$, $\II_\lambda$, or $\III$. Let $e_1, e_2,e_3 \in \q$ denote the following vectors:
    \begin{center}
        \begin{tabular}{c|ccc}
            $\q$ & $e_1$ & $e_2$ & $e_3$ \\ \hline
            $\I$ & $(H,0)$ & $(E,0)$ & $(F,0)$ \\ 
            $\II_\lambda$ & $(H,\lambda H)$ & $(E,0)$ & $(0,F)$ \\ 
            $\III$ & $(H,0)$ & $(E,0)$ & $(0, H + F)$
        \end{tabular}
    \end{center}
    Let $\varepsilon^1, \varepsilon^2, \varepsilon^3 \in \Lambda^{1,0} \g^*$ denote the dual basis of $e_1,e_2,e_3$.
    In each case, we find that $\sum[e_i, \sigma e_i] = 0$.  Thus, \cite[Lemma 1]{AndradaV} implies that $\omega = \i \sum \varepsilon^i \wedge \overline{\varepsilon^i}$ is balanced.

    Next, straightforward computations show that 
    $$\partial \overline \partial (\varepsilon^3 \wedge \overline{\varepsilon^3}) = \begin{cases}
        4 \; \varepsilon^1 \wedge \overline{\varepsilon^1} \wedge \varepsilon^3 \wedge \overline{\varepsilon^3} & \text{if $\q = \I$,} \\
        4 |1 + \lambda|^2\; \varepsilon^1 \wedge \overline{\varepsilon^1} \wedge \varepsilon^3 \wedge \overline{\varepsilon^3} & \text{if $\q = \II_\lambda$,} \\
        4 \; \varepsilon^1 \wedge \overline{\varepsilon^1} \wedge \varepsilon^3 \wedge \overline{\varepsilon^3} & \text{if $\q = \III$.}
    \end{cases}$$
    Therefore, Lemma \ref{lem_pluriclosed_obstruction} implies that $(\SL(2,\C), J)$ does not admit left-invariant pluriclosed metrics.
\end{proof}

\bibliographystyle{amsalpha} 
\bibliography{references}
\end{document}